\documentclass[a4paper, 11pt]{amsart}   	
\usepackage{geometry}                		
\usepackage{graphicx}				
\usepackage{amssymb}

\usepackage{pst-node}
\usepackage{tikz-cd} 
\usepackage{mathtools}

\usepackage{mathrsfs}

\usepackage{url}
\usepackage{hyperref}
\usepackage{amsthm}
\usepackage{amsfonts}
\usepackage{yhmath}
\usepackage{cleveref}
\usepackage{caption}

\theoremstyle{definition}
\newtheorem{definition}{Definition}[section]

\theoremstyle{remark}
\newtheorem{remark}[definition]{Remark}

\theoremstyle{theorem}
\newtheorem{theorem}[definition]{Theorem}

\theoremstyle{corollary}
\newtheorem{corollary}[definition]{Corollary}

\theoremstyle{lemma}
\newtheorem{lemma}[definition]{Lemma}

\theoremstyle{example}
\newtheorem{example}[definition]{Example}

\theoremstyle{prop}
\newtheorem{prop}[definition]{Proposition}

\newcommand{\nospaceperiod}{\makebox[0pt][l]{\,.}}
\begin{document}

\title{Stability conditions on product varieties}

\author{Yucheng Liu }

\address{Department of Mathematics, Northeastern University, 360 Huntington Avenue, Boston, MA 02115, USA}
\email{liu.yuche@husky.neu.edu}

\keywords{Bridgeland stability conditions, product of curves, abelian varieties}
\subjclass[2010]{14F05,14J32,18E30}

\maketitle

\begin{abstract}
Given a stability condition on a smooth projective variety $X$, we construct  a family of stability conditions on $X\times C$, where $C$ is a smooth projective curve. In particular, this gives the existence of stability conditions on arbitrary products of curves. The proof uses, by following an idea of Toda, the positivity lemma established by Bayer and Macr\`i and weak stability conditions on the Abramovich-Polishchuk heart of a bounded t-structure in $D(X\times C)$.
\end{abstract}
\section{introduction}
 Motivated by Douglas's work on D-branes and $\Pi$ stability in \cite{douglas2002dirichlet}, Bridgeland introduced a general theory of stability conditions on triangulated categories in \cite{bridgeland2007stability}; the theory was further studied by Kontsevich and Soibelman in \cite{kontsevich2008stability}. In general, stability conditions are very difficult to construct: while we have a very good knowledge in the case of curves and surfaces (see \cite{bridgeland2008stability} and \cite{ABsurfaces}), starting from 3-folds no example was known on varieties of general type or Calabi-Yau varieties in dimension 4 or higher (for Calabi-Yau threefolds, see \cite{MPabelian}, \cite{bayer2016space} and \cite{li2018stability}). In this paper, we solve this problem for product varieties over any algebraically closed field, when one of the two factors is a curve. 
 
 Let $X$ be a smooth projective variety, $C$ be a smooth projective curve, and let $\sigma=(\mathcal{A},Z)$ be a stability condition on the bounded derived category of coherent sheaves $D(X)$.
 \begin{theorem}\label{main theorem}
 	Assume that the image of the central charge $Z$ is discrete. Then there exists a continuous family of stability conditions on $D(X\times C)$, parametrized by $\mathbb{R}_{>0}\times \mathbb{R}_{>0}$, associated with $\sigma$.

\end{theorem}
 Theorem \ref{main theorem} holds more generally when $D(X)$ is replaced by an admissible subcategory $\mathcal{D}\subset D(X)$ and $D(X\times C)$ is replaced by the base change category $\mathcal{D}_C$. Special cases in dimension three were studied in \cite{Koseki}. 
 
 As a consequence of Theorem \ref{main theorem}, we provide the construction of stability conditions on arbitrary products of curves.
 
 \begin{corollary}\label{main corollary}
 	Let $C_1,\cdots ,C_n$ be smooth projective curves. Then stability conditions exist on $D(C_1\times \cdots\times C_n)$. 
 \end{corollary}

In the case when $n=3$, some related results appeared in \cite{Sun19inequality} and \cite{Sun19stability} when this paper was posted. The techniques are completely different.
 
 An important special case of Corollary \ref{main corollary} is when all curves are elliptic curves. This gives examples of stability conditions on Calabi-Yau varieties of any dimension. In this case, the mirror symmetry version of this statement, for Fukaya categories of products of elliptic curves has been announced by Kontsevich in \cite{Kontsevichstability}.

 There are three main ingredients in the proof. The first one is a construction by Abramovich and Polishchuk in \cite{APsheaves} and \cite{polishchuk2007constant} of a heart of bounded t-structure on $D(X\times S)$, where $S$ is any quasi-projective variety of finite type. We then define a weak stability condition on this category by using a polynomial function naturally associated to $Z$ and Abramovich-Polishchuk's heart. Finally, we use the idea of Toda, studied further by Bayer, Macr\`i and Nuer, and the Positivity Lemma from \cite{bayer2014projectivity} to show a quadratic inequality for stable objects with respect to this weak stability condition.

In Section \ref{Support}, we will establish quadratic inequalities inductively to prove that the stability conditions we constructed satisfy the support property. These quadratic inequalities are stronger than the quadratic inequalities we used in the construction,  and can be viewed as a generalization of Bogomolov-Gieseker inequality for product varieties in any dimension.

\subsection*{Outline of this paper} In Section \ref{Section 2}, we review the definition of weak stability conditions.
In Section \ref{Section 3}, we introduce Abramovich and Polishchuk's construction of global heart and construct global weak stability conditions and polynomial functions associated with it.
In Section \ref{Section 4}, we present the proof of our main theorem, without showing the support property, which will be treated in Section \ref{Support}.

\subsection*{Acknowledgements} I would like to thank my supervisor Emanuele Macr\`i for his patient guidance and advice throughout the process of writing this paper. I am also very grateful to Arend Bayer,  Aaron Bertram, Chunyi Li, Paolo Stellari, Yukinobu Toda and Xiaolei Zhao for helpful discussions and suggestions. I am very grateful to Alex Perry for his careful reading of a preliminary version of this paper and many useful suggestions. I am also very grateful to the referee for a very careful reading of the manuscript and many useful comments. The final write-up of this paper was done while the author was visiting University of Paris-Sud, whose hospitality is gratefully acknowledged. This work was partially supported by the NSF grant DMS-1700751 (PI: Macr\`i).

\subsection*{Notation and Conventions} In this paper, all varieties are integral algebraic varieties over an algebraically closed field $k$, a curve is such a variety of dimension 1. We will use $D(X)$ rather than the usual notation $D^b(CohX)$ to denote the bounded derived categories of coherent sheaves on $X$. We set $\mathbb{H}=\{z\in\mathbb{C}\mid Im(z)>0\}$. All functors are derived unless otherwise specified. We set $Im(z)$, $Re(z)$, and $Arg(z)$ to be the imaginary part, the real part, and the argument of a complex number $z$ respectively.

\section{Stability conditions}\label{Section 2}

In this section, we review the definition and some basic results on weak stability conditions (See \cite{bridgeland2008stability}, \cite{kontsevich2008stability} and \cite{baye2011bridgeland}).
\begin{definition}
	A slicing on a triangulated category $\mathcal{D}$ consists of full subcategories $\mathcal{P}(\phi)\subset\mathcal{D}$ for each $\phi\in\mathbb{R}$, satisfying the following axioms:

	\par
	(a) for all $\phi \in \mathbb{R}$, $\mathcal{P}(\phi+1)=\mathcal{P}(\phi)[1]$,
	
	\par
	(b) if $\phi_1>\phi_2$ and $A_j\in\mathcal{P}(\phi_j)$ then $Hom_{\mathcal{D}}(A_1,A_2)=0$,
	
	\par

	(c) for every $0\neq E\in\mathcal{D}$ there is a sequence of real numbers
	
	$$\phi_1>\phi_2>\cdots>\phi_m$$and a sequence of morphisms 
	
	$$0=E_0\xrightarrow{f_1}E_1\xrightarrow{f_2} \cdots \xrightarrow{f_m}E_m=E $$such that the cone of $f_j$ is in $\mathcal{P}(\phi_j)$ for all $j$.
	
\end{definition}

\begin{definition}\label{first WSC}
	Let $\mathcal{D}$ be a triangulated category and $K(\mathcal{D})$ be its Grothendieck group. A weak stability condition on $\mathcal{D}$ consists of a pair $(\mathcal{P},Z)$, where $\mathcal{P}$ is a slicing and $Z:K(\mathcal{D})\xrightarrow{v}\Lambda\xrightarrow{g}\mathbb{C}$ is a group homomorphism factoring through a lattice $\Lambda$ of finite rank. The pair should satisfy the following conditions:
	\par

	(a) If $0\neq E\in\mathcal{P}(\phi)$ then $Z(E)=m(E)exp(i\pi\phi)$ for some $m(E)\in \mathbb{R}_{\geq 0}$.
	
	\par 
	
		(b) (Support property) There exists a quadratic form $Q$ on $\Lambda\otimes\mathbb{R}$ such that $Q|_{ker(g)}$ is negative definite, and $Q(v(E))\geq 0$, for any object $E\in\mathcal{P}(\phi)$. 
	
\end{definition}

\begin{remark}

If we require $m(E)$ to be strictly positive in (a), then the pair $(\mathcal{P},Z)$ is called a stability condition. By \cite[Lemma 2.2]{bridgeland2008stability}, there is a $S^1$ action on the space of stability conditions. Specifically, for any element $e^{i\theta}\in S^1$, $e^{i\theta}\cdot (Z,\mathcal{P})=(Z',\mathcal{P}')$ by setting $Z'=e^{i\theta} Z$ and $\mathcal{P}'(\phi)=\mathcal{P}(\phi-\theta)$.
\end{remark}

There is an equivalent way of defining a stability condition, which will be more frequently used in this paper. Firstly, we need to define what is a (weak) stability function $Z$ on an abelian category $\mathcal{A}$.
\begin{definition} Let $\mathcal{A}$ be an abelian category. We call a group homomorphism $Z:K(\mathcal{A})\rightarrow \mathbb{C}$ a weak stability function on $\mathcal{A}$ if, for $E\in \mathcal{A}$, we have $Im(Z(E))\geq0$, with $Im(Z(E))=0 \implies  Re(Z(E))\leq0$. 
	If moreover, for $E\neq 0,$ $Im(Z(E))=0\implies Re(Z(E))< 0$, we say that $Z$ is a stability function.
\end{definition}

\begin{definition}\label{slicing }
	A weak stability condition on $\mathcal{D}$ is a pair $\sigma=(\mathcal{A},Z)$ consisting of the heart of a bounded t-structure $\mathcal{A}\subset\mathcal{D}$ and a weak stability function $Z:K(A)\rightarrow \mathbb{C}$ such that (a) and (b) below are satisfied:
	
	\par
	(a) (HN-filtration) The function $Z$ allows us to define a slope for any object $E$ in the heart $\mathcal{A}$ by
	
	$$\mu_{\sigma}(E):=\begin{cases} -\frac{Re(Z(E))}{Im(Z(E))}\ &\text{if} \  Im(Z(E))> 0,\\ +\infty &\text{otherwise.} \end{cases}$$

	The slope function gives a notion of stability: A nonzero object $E\in \mathcal{A}$ is $\sigma$ semi-stable if for every proper subobject $F$, we have $\mu_{\sigma}(F)\leq\mu_{\sigma}(E)$.
	
We require any object $E$ of $\mathcal{A}$ to have a Harder-Narasimhan filtration in $\sigma$ semi-stable ones, i.e., there exists a unique filtration $$0=E_0\subset E_1 \subset E_2\subset \cdots \subset E_{m-1} \subset E_m=E$$ such that $E_i/E_{i-1}$ is  $\sigma$ semi-stable and $\mu_{\sigma}(E_i/E_{i-1})>\mu_{\sigma}(E_{i+1}/E_i)$ for any $1\leq i\leq m$.
	
	(b) (Support property) Equivalently as in Definition \ref{first WSC}, the central charge $Z$ factors as $K(\mathcal{D})\xrightarrow{v} \Lambda\xrightarrow{g} \mathbb{C}$, and there exists a quadratic form $Q$ on $\Lambda_{\mathbb{R}}$ such that $Q|_{ker(g)}$ is negative definite and $Q(v(E))\geq 0$ for any $\sigma$ semi-stable object $E\in\mathcal{A}$.
\end{definition}

\begin{remark}
Similarly, we call $(\mathcal{A},Z)$ a stability condition if $Z$ is a stability function on $\mathcal{A}$. Sometimes, we call the pair $(\mathcal{A},Z)$ a weak pre-stability condition when they just satisfy condition (a).

 If $Z$ has discrete image in $\mathbb{C}$, and $\mathcal{A}$ is Noetherian, then condition (a) is satisfied automatically. See  \cite[Proposition 2.4]{bridgeland2007stability}.
 
\end{remark}
There is an important operation called tilting with respect to a torsion pair, which is very useful for constructing stability conditions.
\begin{definition}
	A torsion pair in an abelian category $\mathcal{A}$ is a pair of full subcategories $(\mathcal{T},\mathcal{F})$ of $\mathcal{A}$ which satisfy $Hom_{\mathcal{A}}(T,F)=0$ for $T\in\mathcal{T}$ and $F\in\mathcal{F}$, and such that every object $E\in\mathcal{A}$ fits into a short exact sequence $$0\rightarrow T\rightarrow E\rightarrow F\rightarrow 0$$
	for some pair of objects $T\in\mathcal{T}$ and $F\in\mathcal{F}$.
\end{definition}
\begin{remark}
	In this paper, most torsion pairs are coming from weak stability conditions $\sigma=(\mathcal{A},Z)$. In fact, let $$\mathcal{T}=\{E\in\mathcal{A}\mid \mu_{\sigma,min}(E)>0\}\ and \ \mathcal{F}=\{E\in\mathcal{A}\mid \mu_{\sigma,max}(E)\leq 0\}$$ be a pair of full subcategories, where $\mu_{\sigma,min}(E)$ is the slope of the last HN-factor of $E$ and $\mu_{\sigma,max}(E)$ is the slope of the first HN-factor of $E$. It is easy to see this is a torsion pair.
\end{remark}
\begin{lemma}[{\cite[Proposition 2.1]{happel1996tilting}}]
	Suppose $\mathcal{A}$ is the heart of a bounded t-structure on a triangulated category $\mathcal{D}$, $(\mathcal{T},\mathcal{F})$ is a torsion pair in $\mathcal{A}$. Then $\mathcal{A}^{\#}=\langle\mathcal{T},\mathcal{F}[1]\rangle$ is a heart of a bounded t-structure on $\mathcal{D}$.
\end{lemma}

 In this paper, we are interested in the case when $\mathcal{D}$ is the bounded derived category of coherent sheaves on an algebraic variety $X$ or an admissible component of it. From now on, $X$ will be a smooth projective variety over an algebraically closed field $k$, and $D(X)$ will be the bounded derived category of coherent sheaves on $X$.

\section{Sheaf of t-structures and polynomial functions}\label{Section 3}

 Suppose there exists a stability condition $(\mathcal{A},Z)$ on $D(X)$, where we assume $\mathcal{A}$ is Noetherian, and the image of $Z$ is discrete.

 Let $S$ be a quasi-projective variety of finite type, and $\mathcal{O}(1)$ be an ample line bundle on $S$. Abramovich and Polishchuk defined a sheaf of t-structures and a global heart $\mathcal{A}_S$ for $D(X\times S)$ in their papers \cite{APsheaves} and \cite{polishchuk2007constant}. Here we summarize some of their beautiful properties we need in the next.

(1) The global heart $\mathcal{A}_S$ is independent of the choice of ample line bundle.

(2) If $S$ is projective, then 	$$\mathcal{A}_S=\{E\in D(X\times S)\mid p_*(E \otimes q^*(\mathcal{O}(n)))\in \mathcal{A} \ for\ all \ n\gg 0\},$$ where $p,q$ are projections from $X\times S$ to $X$ and $S$ respectively.

(3) The functor $p^*:D(X)\rightarrow D(X\times S)$ is t-exact, where $p$ is the projection from $X\times S$ to $X$.

(4) For every closed immersion $i_T: T\xhookrightarrow {}S$, the functor $i_{T*}:D(X\times T)\rightarrow D(X\times S)$ is t-exact and $i_T^*$ is t-right exact.

(5) The heart $\mathcal{A}_S$ is Noetherian.

We also need some definitions from \cite[Section 3]{APsheaves}.

\begin{definition}
	We call an object $E\in \mathcal{A}_S$ to be $S$-torsion if it is the push forward of an object $E'\in D(X\times T)$ for some closed subscheme $T\subset S$.
	
	An object $E\in\mathcal{A}_S $ is torsion free with respect to a closed subscheme $T$ if it contains no nonzero torsion subobject supported on $T$. In this case we say that $E$ is $T$-torsion free.
	
	We say that $E$ is torsion free if it contains no torsion subobject, i.e., it is torsion free with respect to any closed subscheme in $S$.
\end{definition}
\begin{definition}
	The object $E\in\mathcal{A}_S$ is called t-flat if $E_s\in\mathcal{A}$ for arbitrary closed point $s\in S$.
\end{definition}

In the construction of the global heart $\mathcal{A}_S$, the most important case is when $S$ is $\mathbb{P}^r$. In \cite[Section 2.3]{APsheaves}, Abramovich and Polishchuk use Koszul complex to decompose $D(X\times \mathbb{P}^r)$ and construct a global t-structure on it. The Koszul complex can be expressed as follows

$$0\rightarrow \mathcal{O}_{\mathbb{P}^r}\rightarrow \Lambda^r V\otimes \mathcal{O}_{\mathbb{P}^r}(1)\rightarrow\cdots\rightarrow V\otimes \mathcal{O}_{\mathbb{P}^r}(r)\rightarrow \mathcal{O}_{\mathbb{P}^r}(r+1)\rightarrow 0, $$where $V=H^0(\mathbb{P}^r,\mathcal{O}_{\mathbb{P}^r}(1))$. It is not only useful in decomposing derived categories, it is also numerically interesting. Indeed, since the dimensions of $\Lambda^iV$ are binomial coefficients, Koszul complex implies a polynomial structure of $Z(p_*(E\otimes q^* (\mathcal{O}(n)))$ for any $E\in\mathcal{A}_{\mathbb{P}^r}$.

Motivated by this observation, we are able to construct  a global weak stability condition on $D(X\times S)$ for any projective variety $S$ of finite type.

\begin{theorem}\label{glabal weak stablity condition}
	For any smooth projective variety $S$ of finite dimension $r$, we define $(\mathcal{A}_S, Z_S)$ as below:
	\\
	$$\mathcal{A}_S=\{E\in D(X\times S)\mid p_*(E \otimes q^*(\mathcal{O}(n)))\in \mathcal{A}\  for\ all\ n\gg0 \}$$
	$$Z_S(E)=\lim_{n\rightarrow +\infty}\frac{Z(p_*(E\otimes q^*(\mathcal{O}(n)))r!}{n^r vol{(\mathcal{O}(1))}},$$where $vol(\mathcal{O}(1))$ is the volume of $\mathcal{O}(1)$. Then this pair is a weak pre-stability condition on $D(X\times S)$.
\end{theorem}

\begin{proof}
	
	It is easy to see that $(\mathcal{A}_S,Z_S)$ does not change if we change $\mathcal{O}(1)$ to $\mathcal{O}(N)$ for $N\in\mathbb{N}_{>0}$, so we can assume that $\mathcal{O}(1)$ is very ample.	
	
	The definition of $\mathcal{A}_S$ is just taken from \cite{polishchuk2007constant}. We need to prove that $Z_S$ is a weak stability function on $\mathcal{A}_S$.

Suppose $E\in\mathcal{A}_S$ and set $L_E(n):=Z(p_*(E\otimes q^*(\mathcal{O}(n)))$, we claim that $L_E(n)$ is a polynomial of degree no more than $r$, and its leading coefficient lies in $\mathbb{H}\cup \mathbb{R}_{\leq 0}$ for $n\gg 0$. We will prove it by induction on $r$. When $r=0$, the claim is obvious. 
	
	We assume the claim is true for $r\leq i-1$, then prove it for $r=i$. As $k$ is algebraically closed, we can take a general smooth divisor $H\in|\mathcal{O}(1)|$.   Since $\mathcal{A}_S$ is Noetherian, as in \cite[Corollary 3.1.3]{APsheaves} we have the following exact sequence 
	
	$$0\rightarrow F\rightarrow E\rightarrow \bar{E}\rightarrow 0,$$where $F$ is the maximal torsion subobject of $E$ supported over $H$, and $\bar{E}$ is $H$-torsion free. By induction, $L_F(n)$ is a polynomial of degree strictly less than $i$. Therefore, we can assume $E$ is torsion free with respect to $H$.
	
	By the sequence 
	$$0\rightarrow\mathcal{O}(n-1)\rightarrow\mathcal{O}(n)\rightarrow\mathcal{O}(n)|_H\rightarrow0,$$
	we have the following exact sequence 
	$$0\rightarrow q^*\mathcal{O}(n-1)\rightarrow q^*\mathcal{O}(n)\rightarrow q^*\mathcal{O}(n)|_H\rightarrow0$$
	by flatness of $q$, which gives us a triangle 	$$p_*(E\otimes q^*\mathcal{O}(n-1))\rightarrow p_*(E\otimes q^*\mathcal{O}(n))\rightarrow p_*(E\otimes q^*c_*c^*\mathcal{O}(n))\xrightarrow{[1]}p_*(E\otimes q^*\mathcal{O}(n-1))[1]$$where $c:H\rightarrow S$ is the natural inclusion. Note that we have the following commutative diagram
	
	\[ \begin{tikzcd}
	H\times X \arrow{r}{c\times id} \arrow[swap]{d}{q|_H} & S\times X  \arrow{d}{q} \\%
	H \arrow{r}{c}& S\nospaceperiod
	\end{tikzcd}
	\]

	By derived flat base change and projection formula (see \cite[Section 3.3]{huybrechts2006fourier}), we know that $p_*(E\otimes q^*c_*c^*\mathcal{O}(n))=p_*(E\otimes(c\times id)_*q|_H^*c^*\mathcal{O}(n))=p_*(c\times id)_*((c\times id)^*E\otimes q|_H^*c^*\mathcal{O}(n))$. 
	Since $E$ is $H$-torsion free, by \cite[Corollary 3.1.3]{APsheaves}, we have $(c\times id)^*E\in \mathcal{A}_H$.
	
	Therefore, this triangle is a short exact sequence in $\mathcal{A}$ for $n$ sufficiently large. We get $L_E(n)-L_E(n-1)=L_{E|_H}(n)$ where $E|_H\coloneqq (c\times id)^*E$. By induction, $L_{E|_H}(n)$ is a polynomial of degree not bigger than $i-1$, so the degree of $L_E(n)$ is not bigger than $i$. The leading coefficient of $L_{E|_H}(n)$ is an integral multiple of the leading coefficient of $L_E(n)$, therefore, they both lie in $\mathbb{H}\cup \mathbb{R}_{\leq 0}$ by induction. So we proved that $Z_S$ is a weak stability function.
	
	Now for HN-filtration, we know that $\mathcal{A}_S$ is Noetherian by our assumption and \cite[Theorem 3.3.6]{polishchuk2007constant}. Then it suffices to prove that the image of $Z_S$ is discrete. This can be done similarly by induction on the dimension of $S$. When $dim(S)=0$, it is our assumption that image of $Z$ is discrete. From the equation $L_E(n)-L_E(n-1)=L_{E|_H}(n)$, we get $Z_S(E)=Z_H(E|_H)$, hence the inductive step holds. Therefore, $(\mathcal{A}_S,Z_S)$ is a weak pre-stability condition.
	
\end{proof}

\begin{remark}
	The proof and definition of the weak stability function is similar to the way we define the Hilbert polynomial and take its leading exponential coefficient. For instance,
	if we take $X=Spec(\mathbb{C})$, $\mathcal{A}$ is the category of $\mathbb{C}$-vector spaces and $Z(V)=i\cdot dim(V)$ for any finite dimensional $\mathbb{C}$ vector space, then the global heart is the category of coherent sheaves on $S$, and $L_E(n)=i\cdot Hilb_E(n)\ for\ n\gg 0$. 
	
	The construction of weak pre-stability conditions can be generalized to the case when $S$ is a quasi-projective variety of finite type, but we will not use the quasi-projective case in this paper. Using the methods in Section \ref{Support}, we can prove that $(\mathcal{A}_S, Z_S)$ satisfy the support property with respect to the same lattice $\Lambda$ and quadratic form $Q$ as $(\mathcal{A},Z)$ do, but this will not be used in this paper either.
	
\end{remark}

 \begin{corollary}\label{Independence corollary} Let $S$ be a smooth projective variety of finite type.
 	
 	(a) If $E=p^*F\otimes q^*L$, where $F\in\mathcal{A}$ and $L$ is an arbitrary line bundle over $S$, then $Z_S(E)=Z(F)$. Moreover, $Z_S$ is independent of the choice of ample line bundle $\mathcal{O}(1)$. 
 	\par 
 	(b) If $E\in\mathcal{A}_S$ and $\{s\in S|i_s^*E\in \mathcal{A} \ and\ i_s^*E\neq 0\}$ contains an open dense subset in $S$, where $i_s:X\times\{s\}  \rightarrow X\times S$ is the natural inclusion, then $Z_S(E)\neq 0$.
 
 \end{corollary}
 \begin{proof}
 	We deal with the untwisted case $E=p^*F$ first. The following equation hold:
 	\begin{equation*}
 	\begin{split}
 	Z_S(p^*F) & = \lim_{n\rightarrow +\infty}\frac{Z(p_*(p^*F\otimes q^*(\mathcal{O}(n))))r!}{n^rvol(\mathcal{O}(1))} \\ & = \lim_{n\rightarrow +\infty}\frac{Z(F\otimes H^0(S,\mathcal{O}(n)))r!}{n^rvol(\mathcal{O}(1))} \\ & = \lim_{n\rightarrow +\infty}\frac{Z(F)\frac{Vol(\mathcal{O}(1))}{r!} n^rr!}{n^rvol(\mathcal{O}(1))} \\ & = Z(F).
 	\end{split}
 	\end{equation*}
 	
 	The second equation comes from projection formula, and the third equation follows from Asymptotic Riemann-Roch and Serre vanishing (see \cite[Corollary 1.1.25]{lazarsfeld2004positivity}). It is easy to see that twisting $p^*F$ by $q^*L$ will not effect this equation. This proves the first half of (a).

 	For the independence of $Z_S$ on the choice of $\mathcal{O}(1)$, we know the line bundles generate $K(D(S))$ for $S$ smooth. Hence, the objects of the form $p^*F\otimes q^*L$ span the group $K(D(X\times S))$. Therefore, $Z_S$ is determined by its value on $p^*F\otimes q^*L$. Since $Z_S( p^*F\otimes q^*L)=Z(F)$ is independent of the choice of $\mathcal{O}(1)$, we proved the independence of $Z_S$ on the choice of $\mathcal{O}(1)$.

 	For (b), similarly as in previous theorem, we have the following sequence 
 	$$0\rightarrow F\rightarrow E\rightarrow \bar{E}\rightarrow 0,$$where $F$ is the maximal torsion subobject of $E$, and $\bar{E}$ is torsion free. It is easy to see that $\{s\in S|i_s^*\bar{E}\in\mathcal{A}\ and\ i_s^*\bar{E}\neq 0\}$ contains an open dense subset in $S$. Therefore, we can assume $E$ is torsion free. Since $Z_S$ is independent of the choice of the ample line bundle, we can choose $\mathcal{O}(1)$ to be globally generated.  Because of the smoothness of $S$ we are able to find a smooth divisor $D$ in the linear system $|\mathcal{O}(1)|$ such that $D\cap \{s\in S|i_s^*E\in\mathcal{A}\ and\ i_s^*E\neq 0\}$ is open and dense in $D$. Since $E$ is torsion free, we have $E|_D\in \mathcal{A}_D$, and we know that the leading coefficient of $L_E(n)$ is the leading coefficient of $L_{E|_D}(n)$ times a nonzero constant (the constant is the reciprocal of dimension of $S$, which implies $Z_S(E)=Z_D(E|_D)$). This finishes the proof by induction.
 \end{proof}

\begin{remark}
	Though $Z_S$ is independent of the choice of ample line bundle $\mathcal{O}(1)$, $L_E(n)$ is definitely dependent of the choice of ample line bundle $\mathcal{O}(1)$. We suppress this dependence in our notation for simplicity.
\end{remark}

From the polynomial $L_E(n)$, we have two ways to define a slope of an object $E\in\mathcal{A}_S$.

(1) The first one only cares about $Z_S(E)$. We define $\mu_1(E)$ as
$$\mu_1(E)\coloneqq\begin{cases} -\frac{Re(Z_S(E))}{Im(Z_S(E))} &\text{if} \ Im(Z_S(E))> 0,\\ +\infty &\text{otherwise.}\end{cases}$$

(2) The second one is the slope of the first nonzero coefficient of $L_E(n)$:
$$\mu_2(E)\coloneqq\begin{cases} -\ \underset{n\rightarrow +\infty}{\lim}\frac{Re(L_E(n))}{Im(L_E(n))} &\text{if it is well defined,}\\ +\infty  &\text{otherwise.}\end{cases}$$

We use $$\psi(E)=\frac{-cot^{-1}(\mu_2(E))}{\pi}$$ to denote the phase of $E$. Then $0<\psi(E)\leq 1$ for $E\in \mathcal{A}_S$.
\begin{remark}
	 In the second case, unlike the usual slope function, any subobjects have smaller or equal slope is not equivalent to that any quotient objects have bigger or equal slope. Therefore, we define $E$ to be semi-stable with respect to $\mu_2$, if for any subobject $F\subset E$, we have $\mu_2(F)\leq \mu_2(E)$ and for any quotient objects $G$ of $E$, we have $\mu_2(G)\geq \mu_2(E)$.
\end{remark}

The semi-stability with respect to $\mu_2$ is closely related to the slicing constructed in \cite[Section 4]{bayer2014mmp}.

More specifically, given a stability condition $\sigma=(Z, \mathcal{P})$ on $D(X)$ and a phase $\phi\in\mathbb{R}$, then we have its associated t-structure $\mathcal{P}(>\phi)=\mathcal{D}^{\leq -1}$, $\mathcal{P}(\leq\phi)=\mathcal{D}^{\geq 0}$. By Abramovich and Polishchuk's construction, we get $\mathcal{P}_S(>\phi)$, $\mathcal{P}_S(\leq \phi)$ as a t-structure on $D(X\times S)$. Then we have the following lemma in \cite[Section 4]{bayer2014mmp}.

\begin{lemma}[{\cite[Lemma 4.6]{bayer2014mmp}}]\label{global slicing}
	Assume $\sigma=(Z, \mathcal{P})$ is a stability condition as in our setup, and $\mathcal{P}_S(>\phi)$, $\mathcal{P}_S(\leq \phi)$ defined as above. There is a slicing $\mathcal{P}_S$ on $D^b(X\times S)$ defined by
	
	$$\mathcal{P}_S(\phi)=\mathcal{P}_S(\leq \phi)\cap \underset{\epsilon>0}{\cap}\mathcal{P}_S(>\phi-\epsilon).$$
\end{lemma}

\begin{lemma}\label{equal}
	If $E\in \mathcal{P}_S(\phi)$, then $\psi(E)=\phi$.

\end{lemma}
\begin{proof}
	By definition of $\mathcal{P}_S(\phi)$ in Lemma \ref{global slicing}, we know that for $n\gg 0$, the phases of the HN-factors of $p_*(E\otimes q^*(\mathcal{O}(n)))$ lie in $(\phi-\epsilon,\phi]$. Hence we have $Arg(L_E(n))\in((\phi-\epsilon)\pi,\phi\pi]$ for $n\gg 0$. This implies $\phi\geq\psi(E)\geq \phi-\epsilon$ for arbitrary $\epsilon>0$. Therefore, $\psi(E)=\phi$.
\end{proof}

Then, we have the following proposition.

\begin{prop}\label{equivalenc}
	Suppose $E\in \mathcal{A}_S$, then $E\in\mathcal{P}_S(\phi)$ if and only if $E$ is semi-stable of phase $\phi$ with respect to $\mu_2$.
\end{prop}

\begin{proof}
	If $E$ is semi-stable of phase $\phi$ with respect to $\mu_2$, then take the HN-filtration of $E$ with respect to the slicing. We get $$0=E_0\subset E_1 \subset E_2\subset \cdots \subset E_{n-1} \subset E_n=E,$$ where $E_{i}/E_{i-1}\in\mathcal{P}_S(\phi_i)$ and $$\phi_1>\phi_2>\cdots>\phi_n.$$
	
	Since $E$ is semi-stable with respect to $\mu_2$, we get $\phi_1\leq \phi$ and $\phi_n\geq\phi$. Therefore, $E\in \mathcal{P}_S(\phi)$.
	
	Assume $E\in\mathcal{P}_S(\phi)$ and is not semistable with respect to $\mu_2$. Let $Q$ be a quotient object of $E$ in $\mathcal{A}_S$ with $\psi(Q)< \psi(E)=\phi$ by Lemma \ref{equal}. Take the HN-filtration $$0=Q_0\subset Q_1 \subset Q_2\subset \cdots \subset Q_{n-1} \subset Q_n=Q,$$ of $Q$ with respect to the slicing, where $Q_{i}/Q_{i-1}\in\mathcal{P}_S(\phi_i)$ . Then we have $\phi_n\leq\psi(Q)<\phi$ by see-saw principle. Therefore the  nontrivial morphism $E\rightarrow Q\rightarrow Q/Q_{n-1}$ contradicts the definition of slicing. 
	Similarly, we can draw a contradiction for subobject case.
\end{proof}

\begin{corollary}
	If $E\in\mathcal{A}_S$, then $E$ admits HN filtration with respect to $\mu_i$ for $i=1,2$.
\end{corollary}
\begin{proof}
	For $\mu_1$, this follows from the fact that $\mathcal{A}_S$ is Noetherian and the image of $Z_S$ is discrete.
	
	For $\mu_2$, it follows from Proposition \ref{equivalenc}.
\end{proof}

\begin{lemma}\label{constant central charge}
	If $S$ is a smooth projective variety of finite type and $E\in\mathcal{A}_S$ is t-flat over $S$, then $Z_S(E)=Z(E_s)$ for any point $s\in S$.
\end{lemma}
\begin{proof}
	We can prove it by induction on the dimension of $S$. If the dimension of $S$ is 0, then the statement is trivial. Now for the inductive step, we use the same argument in the proof of part (b) in Corollary \ref{Independence corollary}. Indeed, for any point $s\in S$, there exists a smooth divisor $D$ such that $s\in D$, and $Z_S(E)=Z_D(E|_D)$. Hence, $Z(E_s)=Z_D(E|_D)=Z_S(E)$ by induction.
\end{proof}

\begin{prop}\label{fiberwise semistable}
	Let $S$ be a smooth projective variety of finite type and $E\in \mathcal{A}_S$ be $t$-flat. Consider the following propositions.

	(1) $E$ is semi-stable of phase $\phi$ with respect to $\mu_2$.

	(2) $E_s\in\mathcal{P}(\phi)$, for arbitrary $s\in S$.
	
	Then (1) implies (2).
\end{prop}
\begin{proof}
	If $E=0$, the statement is obvious. 
	
	Now assume $E\in\mathcal{A}_S$ is a nonzero object and t-flat, then we can deduce that $Z_S(E)\neq 0$. Otherwise we have $Z(E_s)=0$, which implies $E_s$ is the zero object in $\mathcal{A}$ for all $s\in S$. Hence $E=0$, which contradicts our assumption.  Now we have $Z(E_s)\neq 0$ and the phase of $E_s$ is also $\phi$ by Lemma \ref{constant central charge}. 
	
	On the other hand, the assumption that $E$ is semi-stable of phase $\phi$ with respect to $\mu_2$ is equivalent to $$E\in\mathcal{P}_S(\phi)=\mathcal{P}_S(\leq \phi)\cap \underset{\epsilon>0}{\cap}\mathcal{P}_S(>\phi-\epsilon)$$ by Lemma \ref{global slicing} and Proposition \ref{equivalenc}. Since $i_s^*$ is t-right exact, it sends objects in $\mathcal{P}_S(>\phi-\epsilon)$ to objects in $\mathcal{P}(>\phi-\epsilon)$, hence $E_s\in\mathcal{P}(>\phi-\epsilon)$. The assumption that $E$ is t-flat implies $E_s\in\mathcal{A}$. Hence we have $E_s\in\mathcal{P}(>\phi-\epsilon)\cap \mathcal{P}(\leq 1)$ for all $\epsilon>0$. 
	
	Combing these two facts, we get $E_s\in \mathcal{P}(\phi)$.

\end{proof}

\begin{prop}\label{semistable reduction}
	If  $S$ is a smooth projective variety of finite type, and $E\in \mathcal{A}_S$ is semi-stable with respect to $\mu_1$ of phase $\phi$ and $Z_S(E)\neq 0$, then there exists a short exact sequence 
	
	$$0\rightarrow K\rightarrow E\rightarrow Q\rightarrow 0$$
	such that $K\in \mathcal{P}_S(\phi)$, $Q\in \mathcal{P}_S(<\phi)$ and $Z_S(Q)=0$, where $Q$ could be zero.
\end{prop}
\begin{proof}
The sequence comes from the HN filtration of $E$ with respect to $\mu_2$, or equivalently, the global slicing $\mathcal{P}_S$. Indeed, suppose $$0=E_0\subset E_1 \subset E_2\subset \cdots \subset E_{n-1} \subset E_n=E$$ is the filtration. 

We claim that  $Z_S(E_i/E_{i-1})=0$ for all $i>1$. Otherwise there exists $i_0>1$ such that $Z_S(E_{i_0}/E_{i_0-1})\neq 0$, then $\mu_1(E_{i_0}/E_{i_0-1})=\mu_2(E_{i_0}/E_{i_0-1})$ and $E_1$ would destabilize $E$ with respect to $\mu_1$ by see-saw principle of $Z_S$. Therefore, $Z_S(E_1)=Z_S(E)\neq 0$.

Then $0\rightarrow E_1\rightarrow E\rightarrow E/E_1\rightarrow 0$ is the sequence we need.
\end{proof}

\begin{example}
	Take $X=S=\bold{P}^1$ and $\sigma=(CohX, Z)$, where $Z(E)=-\deg(E)+i\cdot rk(E)$. Then the ideal sheaf $\mathcal{I}$ of a closed point $(x_0, s_0)$ in $\bold{P}^1\times\bold{P}^1$ is an example of an object that is semi-stable with respect to $\mu_1$ but not semi-stable with respect to $\mu_2$. This is because $\mathcal{I}_s$ is semi-stable for every $s\in S$ except $s_0\in S$. Indeed, we have $$\mathcal{I}_s=\begin{cases}
	\mathcal{O}_{\bold{P}^1\times s} & \text{if} \ s\neq s_0, \\ \mathcal{O}_{\bold{P}^1\times s_0}(-1)\oplus k(x_0,s_0)& \text{if} \ s=s_0,
	\end{cases}$$ where $k(x_0,s_0)$ is the skyscraper sheaf of the point $(x_0,s_0)$.
	
	In this case, the sequence in Proposition \ref{semistable reduction} is
	
	$$0\rightarrow \mathcal{O}(0,-1)\rightarrow \mathcal{I}\rightarrow \mathcal{O}_{\bold{P}^1\times s_0}(-1)\rightarrow 0.$$
	
	One can check that $\mathcal{O}(0,-1)$ is of phase $\frac{1}{2}$ with $L_{\mathcal{O}(0,-1)}(n)=i\cdot n$, $\mathcal{O}_{\bold{P}^1\times s_0}(-1)$ is torsion and of phase $\frac{1}{4}$ with $L_{\mathcal{O}_{\bold{P}^1\times s_0}(-1)}(n)=1+i$.
\end{example}

\section{Existence of stability conditions}\label{Section 4}

In the rest of this paper, we will focus on the case when $S$ is a smooth projective curve. The polynomial $L_E(n)$ becomes a linear polynomial, which can be written in the form $$L_E(n):=a(E)n+b(E)+i(c(E)n+d(E)),$$ where $a,b,c,d$ are group homomorphisms from $K(\mathcal{A_S})$ to $\mathbb{R}$. 

By Theorem \ref{glabal weak stablity condition}, we know that $a+ic$ is a weak stability function on $\mathcal{A}_S$, and $L_E(n)$ will lie in  $\mathbb{H} \cup \mathbb{R}_{< 0}$ for nonzero object $E\in\mathcal{A}_S$ and $n\gg0$. By this observation, we have the following lemma.

\begin{lemma}\label{easy lemma} For a nonzero object $E\in\mathcal{A}_S$, we have the following inequalities.
	
	(i) $c(E)\geq 0$.
	
	(ii) If $c(E)=0$, then $d(E)\geq 0$ and $a(E)\leq 0$.
	
	(iii) If $c(E)=a(E)=d(E)=0$, then $b(E)< 0$.
\end{lemma}
\begin{proof}
	This follows easily from the observation and the definition of weak stability function.
\end{proof}

Now we can restate the Positivity Lemma from \cite[Lemma 3.3]{bayer2014projectivity} in terms of $a,b,c,d$.

\begin{lemma}[Restatement of Positivity Lemma]\label{restatement}
If $E\in\mathcal{A}_S$ is t-flat and $E_s$ is semi-stable for any point $s\in S$, then $b(E)c(E)-a(E)d(E)\geq 0$. 
\end{lemma}
\begin{proof}
	It is easy to see that $E_s$ is of slope $\frac{-a(E)}{c(E)}$ for any point $s\in S$ from Lemma \ref{constant central charge}. We rotate $\sigma=(\mathcal{A},Z)$ by angle $\theta$ to make $E_s$ of phase $1$. Since $E_s$ is in the rotated global heart $e^{i\theta}\cdot\mathcal{A}$, then $E$ is in the corresponding global heart $e^{i\theta}\cdot \mathcal{A}_S$ by \cite[Corollary 3.3.3]{APsheaves}. Therefore, $Im(e^{i\theta}L_E(n))\geq 0$ for $n\gg 0$.
	
	Assume a rotation by angle $\theta$ makes $E_s$ of phase $1$, this means  $e^{i\theta}Z(E_s)\in\mathbb{R}_{<0}$. Since $E_s$ is of slope $\frac{-a(E)}{c(E)}$, it is equivalent to $$e^{i\theta}(a(E)+ic(E))=-\sqrt{a(E)^2+c(E)^2}.$$ 
	
Under this assumption, $Im(e^{i\theta}L_E(n))\geq 0$ for $n\gg 0$ is equivalent to 
	 $$Im(e^{i\theta}(b(E)+id(E)))\geq 0$$ since the imaginary part of linear term vanishes after the rotation of angle $\theta$.
	 
	  This implies $$Im(ic(E)-a(E))(b(E)+id(E))\geq 0,$$ which is equivalent to $b(E)c(E)-a(E)d(E)\geq 0$.
	
\end{proof}
Combing Positivity Lemma with Proposition \ref{semistable reduction}, we get the following lemma.
\begin{lemma}\label{refinement}
If $E\in\mathcal{A}_S$ is semi-stable with respect to $\mu_1$, then $b(E)c(E)-a(E)d(E)\geq 0$. 
\end{lemma}
\begin{proof}
If $c(E)=0$, then the inequality follows from Lemma \ref{easy lemma}. 

	Now, we assume $c(E)>0$. By Proposition \ref{semistable reduction}, we have a short exact sequence 
		$$0\rightarrow K\rightarrow E\rightarrow Q\rightarrow 0$$
			such that $K\in P_S(\phi)$, $Z_S(Q)=0$ and $Q\in P_S(<\phi)$. A torsion subobject of $K$ would destabilize $E$ with respect to $\mu_1$,  so $K$ is torsion free. Hence $K$ is t-flat by \cite[Corollary 3.1.3]{APsheaves}. By Proposition \ref{fiberwise semistable}, we can apply Lemma \ref{restatement} on $K$. Moreover, since $Z_S(Q)=0$ and $Q\in P_S(<\phi)$, we have $$a(E)=a(K),\ c(E)=c(K)$$ and $$\frac{-b(Q)}{d(Q)}<\frac{-a(K)}{c(K)}$$ if $Q$ is nonzero. Therefore, $$b(E)c(E)-a(E)d(E)=b(K)c(K)-a(K)d(K)+b(Q)c(K)-a(K)d(Q)\geq 0.$$
	
\end{proof}

For the simplicity of our statements and arguments, we introduce the following definition.
\begin{definition}\label{rational stability condition}
	If $(\mathcal{A},Z)$ is a stability condition, and the image of $Z$ lies in $\mathbb{Q}\oplus \mathbb{Q}i$, we call $(\mathcal{A},Z)$ a rational stability condition. We use $RStab(X)$ to denote the set of rational stability conditions on $D(X)$.
\end{definition}
\begin{remark}
	By \cite[Proposition 5.0.1]{APsheaves}, we know the heart $\mathcal{A}$ of a rational stability condition is Noetherian. In this case, the images of $a,b,c,d$ are rational. We focus on the rational stability conditions just for the simplicity of statements and arguments. All results and proofs in the rest of this paper can be easily adapted to the stability conditions whose central charge have discrete image. 
\end{remark}

We assume that $\sigma=(\mathcal{A},Z)$ is a rational stability condition, then for any positive rational number $t$, we can define the following slope function, coming from the weak stability function $Z_t(E)=a(E)t-d(E)+ic(E)t$:

$$\nu_t(E)=\begin{cases}
\frac{-a(E)t+d(E)}{c(E)t} &\text{if}\ c(E)\neq 0,\\ +\infty &\text{otherwise.}
\end{cases}$$

By part (i) and (ii) of Lemma \ref{easy lemma}, we know that $Z_t$ is a weak stability function on $\mathcal{A}_S$. Since $t$ is a fixed positive rational number, the pair $\sigma_t=(\mathcal{A}_S, Z_t)$ admits HN property because  $\mathcal{A}_S$ is Noetherian and $Z_t$ is discrete.
Then $ \mathcal{A}_S$ can be decomposed into two parts, torsion part $\mathcal{T}=\{E\in \mathcal{A}_S\mid \nu_{t,min}(E)>0\}$ and torsion free part $\mathcal{F}=\{E\in \mathcal{A}_S\mid \nu_{t,max}(E)\leq 0\}$. We can apply tilting method on this heart to get a new heart $\mathcal{A}_S^t=\langle \mathcal{T},\mathcal{F}[1]\rangle$.

\begin{prop}\label{the construction}
	For arbitrary $s,t\in \mathbb{R}_{>0}$, $Z_S^{s,t}(E)=c(E)s+b(E)+i(-a(E)t+d(E))$ is a stability function on $\mathcal{A}_S^t$.
\end{prop}
\begin{proof}
	It is easy to see that $-a(E)t+d(E)\geq 0$ for $E\in\mathcal{A}_S^t$ from the definition of $\mathcal{A}_S^t$. Now we need to prove that if $-a(E)t+d(E)=0$, then $c(E)s+b(E)<0$ for nonzero $E\in\mathcal{A}_S^t$. We have the following short exact sequence
	
	$$0\rightarrow F[1]\rightarrow E\rightarrow T\rightarrow 0$$where $F\in\mathcal{F}, T\in\mathcal{T}$. Therefore, we have to deal with the following two cases.
	
	Firstly, assume $-a(T)t+d(T)=0$. By definition of $\mathcal{T}$, we have $c(T)=0$ and $\nu_t(T)=+\infty$. Therefore, in this case $-a(T)t+d(T)=0$ is equivalent to $a(T)=d(T)=0$ by Lemma \ref{easy lemma}, which also implies $b(T)<0$.
	
	Now we deal with $F$. By definition of $\mathcal{F}$, we know that $c(F)>0$ if $F$ is nonzero. Thus $F\in\mathcal{F}$ and $-a(F)t+d(F)=0$ implies that $F$ is semi-stable with respect to $\sigma_t$. Therefore, it suffices to prove $c(F)s+b(F)>0$ in this case. 
	
	Take $$0=F_0\subset F_1 \subset F_2\subset \cdots \subset F_{l-1} \subset F_l=F$$ as the HN filtration of $F$ with respect to $\mu_1$. We let $Q_k=F_k/F_{k-1}$ be the $k$-th HN factor of $F$, for $1\leq k\leq l$. We have 
	
\begin{equation}
	 \frac{-a(Q_k)}{c(Q_k)}>\frac{-a(Q_{k+1})}{c(Q_{k+1})}
\end{equation}
	 by the property of HN filtration.
	 If $c(Q_1)=0$, then $F_1$ would destabilize $F$ with respect to $\nu_t$. Hence $c(Q_1)>0$, which implies $c(Q_k)>0$ for $1\le k\le l$. Moreover, $c(Q_k)>0$ and $Q_k$ is semi-stable with respect to $\mu_1$ implies $Q_k$ is torsion free, which is equivalent to being t-flat since $S$ is a curve. Applying Lemma \ref{refinement}, we get
	 
	 \begin{equation}
	 b(Q_k)c(Q_k)\geq a(Q_k)d(Q_k)
	 \end{equation}
	 for $1\leq k\leq l$. The last piece of data is that $F$ is semi-stable of slope $0$ with respect to $\nu_t$. We have
	 \begin{equation}
	 \frac{\Sigma_{k=1}^j (-a(Q_k)t+d(Q_k))}{\Sigma_{k=1}^jc(Q_k)t}\leq 0\leq \frac{\Sigma_{k=j}^l(-a(Q_k)t+d(Q_k))}{\Sigma_{k=j}^lc(Q_k)t}
	 \end{equation}
	for any $1\leq j\leq l$.

	Using (1), (2) and (3), we are able to prove the following inequality:
	\begin{equation*}
	\begin{split}
	b(F)=\Sigma_{k=1}^lb(Q_k) &\geq \Sigma_{k=1}^l \frac{a(Q_k)d(Q_k)}{c(Q_k)} \\ & =\frac{a(Q_l)}{c(Q_l)}d(F)-\Sigma_{j=1}^{l-1}\Sigma_{k=1}^jd(Q_k)(\frac{a(Q_{j+1})}{c(Q_{j+1})}-\frac{a(Q_j)}{c(Q_j)} )\\ & \geq \frac{a(Q_l)}{c(Q_l)}a(F)t-\Sigma_{j=1}^{l-1}\Sigma_{k=1}^ja(Q_k)t(\frac{a(Q_{j+1})}{c(Q_{j+1})}-\frac{a(Q_j)}{c(Q_j)} )\\ &  =\Sigma_{k=1}^l\frac{ta(Q_k)^2}{c(Q_k)} \\& \geq 0.
	\end{split}
	\end{equation*}
	The first inequality is from (2) and the fact $c(Q_k)>0$, the second equality is Abel's summation formula. The second inequality comes from (1) and the left side of (3). The last equality is Abel's summation formula.
	
	Therefore $c(F)s+b(F)>0$ for $s\in\mathbb{R}_{>0}$. The proof is complete.
\end{proof}
In fact, we can prove that the pair $(\mathcal{A}_S^t,Z_S^{s,t})$ is a stability condition.

\begin{theorem}\label{map of rational stability conditions}
	If $(\mathcal{A},Z)$ is a rational stability condition on $D(X)$, then the pair $\sigma_{s,t}=(\mathcal{A}_S^t, Z_S^{s,t}) $ is a rational stability condition on $D(X\times S)$ for $s,t\in\mathbb{Q}_{>0}$.
\end{theorem}
\begin{proof}
	Firstly, we need to prove that $\mathcal{A}_S^t$ is Noetherian. The idea of the proof is essentially the same as in \cite{piyaratne2019moduli}. Readers should consult \cite[Section 2.3]{piyaratne2019moduli} for details.
	
	 Suppose there exists an object $E\in\mathcal{A}_S^t$ and an infinite sequence of surjections $$  E\twoheadrightarrow E_1\twoheadrightarrow E_2\twoheadrightarrow\cdots.$$  
	 
	 Since $a,d$ are discrete and $Im(Z_S^{s,t}(F))\geq 0$ for any $F\in \mathcal{A}_S^t$, we may assume $Im(Z_S^{s,t}(E_i))=Im(Z_S^{s,t}(E))$ for all $i$. Then consider the following short exact sequences in $\mathcal{A}_S^t$
	
	$$0\rightarrow F_i\rightarrow E\rightarrow E_i\rightarrow 0.$$ 
	
	We have $Im(Z_S^{s,t}(F_i))=0$ by assumption. By the Noetherianity of $\mathcal{A}_S$, we can assume that $\mathcal{H}_{\mathcal{A}_S}^0(E)=\mathcal{H}_{\mathcal{A}_S}^0(E_i)$ and $\mathcal{H}_{\mathcal{A}_S}^{-1}(F_i)$ is independent of $i$. By setting $V=\mathcal{H}_{\mathcal{A}_S}^{-1}(E)/\mathcal{H}_{\mathcal{A}_S}^{-1}(F_i)$, we have the following short exact sequence in $\mathcal{A}_S$ $$0\rightarrow V\rightarrow \mathcal{H}_{\mathcal{A}_S}^{-1}(E_i)\rightarrow \mathcal{H}_{\mathcal{A}_S}^0(F_i)\rightarrow 0.$$ 
	
	Then we look at the short exact sequences $$0\rightarrow F_i\rightarrow F_j\rightarrow F_{ij}\rightarrow 0,\ for\ i<j.$$  
	
	Since $\mathcal{H}_{\mathcal{A}_S}^{0}(F_{i}), \mathcal{H}_{\mathcal{A}_S}^{0}(F_{ij})\in\mathcal{T}$ and $$Im(Z_S^{s,t}(\mathcal{H}_{\mathcal{A}_S}^{0}(F_{i})))=Im(Z_S^{s,t}(\mathcal{H}_{\mathcal{A}_S}^{0}(F_{ij})))=0,$$ by the proof of Proposition \ref{the construction}, we get $$a(\mathcal{H}_{\mathcal{A}_S}^{0}(F_{i}))=c(\mathcal{H}_{\mathcal{A}_S}^{0}(F_{i}))=d(\mathcal{H}_{\mathcal{A}_S}^{0}(F_{i}))=0$$ and  $$a(\mathcal{H}_{\mathcal{A}_S}^{0}(F_{ij}))=c(\mathcal{H}_{\mathcal{A}_S}^{0}(F_{ij}))=d(\mathcal{H}_{\mathcal{A}_S}^{0}(F_{ij}))=0.$$ Hence, we  have $$c(\mathcal{H}_{\mathcal{A}_S}^{-1}(F_{ij}))=c(\mathcal{H}_{\mathcal{A}_S}^{-1}(F_{j}))-c(\mathcal{H}_{\mathcal{A}_S}^{-1}(F_{i}))=0$$ and $\mathcal{H}_{\mathcal{A}_S}^{-1}(F_{ij})\in \mathcal{F}$ since we assumed that $\mathcal{H}_{\mathcal{A}_S}^{-1}(F_i)$ is independent of $i$.  From the definition of $\mathcal{F}$ we know that $\mathcal{H}_{\mathcal{A}_S}^{-1}(F_{ij})=0$. Therefore, we have $$\mathcal{H}_{\mathcal{A}_S}^{0}(F_{1})\subset \mathcal{H}_{\mathcal{A}_S}^{0}(F_{2})\cdots,$$ which gives us an infinite filtration in $\mathcal{F}$ $$\mathcal{H}_{\mathcal{A}_S}^{-1}(E_{1})\subset \mathcal{H}_{\mathcal{A}_S}^{-1}(E_{2})\cdots$$ where $\mathcal{H}_{\mathcal{A}_S}^{-1}(E_{j})/\mathcal{H}_{\mathcal{A}_S}^{-1}(E_{i})=\mathcal{H}_{\mathcal{A}_S}^{0}(F_{ij})$.
	
	Therefore, by Lemma \ref{easy lemma}, we know that $b(\mathcal{H}_{\mathcal{A}_S}^{0}(F_{ij}))<0$ if $\mathcal{H}_{\mathcal{A}_S}^{0}(F_{ij})\neq 0$. Hence $b(\mathcal{H}_{\mathcal{A}_S}^{-1}(E_{j}))<b(\mathcal{H}_{\mathcal{A}_S}^{-1}(E_{i}))$ for any $i<j$ if $\mathcal{H}_{\mathcal{A}_S}^{0}(F_{ij})\neq 0$, which is equivalent to $F_{ij}\neq 0$. 
	
	 Let us use $Q_i$ to denote $\mathcal{H}_{\mathcal{A}_S}^{-1}(E_{i})$. As in the proof of \cite[Lemma 2.15]{piyaratne2019moduli}, we can assume $Q_1$ is semi-stable with respect to $\sigma_t$ by induction on the number of HN factors of $Q_1$. Hence by \cite[Sublemma 2.16]{piyaratne2019moduli}, $Q_i$ is semi-stable with respect to $\sigma_t$ for all $i$. By Lemma \ref{refinement}, we know that $$b(Q_i)c(Q_i)-a(Q_i)d(Q_i)\geq 0,$$ where $a,d,c$ are constant on $Q_i$ and $b$ decreases as $i$ grows. Since $b$ is discrete, the inequalities hold for all $i$ only if $$b(Q_i)=b(Q_{i+1})$$ for $i\gg 0$, or $$c(Q_i)=0.$$ The first case implies $F_{ij}=0$ for $i$ sufficiently large, the second case combining the fact $Q_i\in\mathcal{F}$ forces $Q_i=0$. In either case, the filtration terminates after finite steps. 
	 
	 By Lemma \ref{support property}, we also have the support property.

\end{proof}

\begin{remark}
	In fact, our construction also works with analogue proofs for stability conditions on Kuznetsov components; please see \cite{bayer2017stability} and \cite{bayer2019stability}.
\end{remark}

We conclude this section by providing a lemma, which might be useful in characterizing geometric stability conditions.

\begin{lemma}\label{geometricity}
	Suppose $E$ is an object in $\mathcal{A}_S$. If $a(E)=c(E)=d(E)=0$ and $b(E)$ is minimal in the image of the real part of $Z$. Then $E$ is a simple object in $\mathcal{A}_S^t$.
\end{lemma}
\begin{proof}
	Since $c(E)=0$, we have $E\in\mathcal{T}$, hence $E\in\mathcal{A}_S^t$. Suppose we have a short exact sequence 
	$$0\rightarrow K\rightarrow E\rightarrow Q\rightarrow 0$$ in $\mathcal{A}_S^t$.  Then taking cohomology with respect to $\mathcal{A}_S$ gives us an exact sequence 
	
	$$0\rightarrow\mathcal{H}_{\mathcal{A}_S}^{-1}(Q)\rightarrow \mathcal{H}_{\mathcal{A}_S}^{0}(K)\rightarrow E\rightarrow\mathcal{H}_{\mathcal{A}_S}^{0}(Q)\rightarrow 0$$ in $\mathcal{A}_S$. By assumption we know that $a(E)=c(E)=d(E)=a(\mathcal{H}_{\mathcal{A}_S}^{0}(Q))=c(\mathcal{H}_{\mathcal{A}_S}^{0}(Q))=d(\mathcal{H}_{\mathcal{A}_S}^{0}(Q))=0$, hence $\mathcal{H}_{\mathcal{A}_S}^{-1}(Q)$ and $\mathcal{H}_{\mathcal{A}_S}^{0}(K)$ are of the same slope with respect to $\nu_t$. This contradicts the definition of $\mathcal{T}$ and $\mathcal{F}$ unless $\mathcal{H}_{\mathcal{A}_S}^{-1}(Q)=0$. Therefore, we have the following short exact sequence $$0\rightarrow \mathcal{H}_{\mathcal{A}_S}^{0}(K)\rightarrow E\rightarrow\mathcal{H}_{\mathcal{A}_S}^{0}(Q)\rightarrow 0$$ in $\mathcal{T}$. Since $b(E)$ is minimal, we know that either $K$ or $Q$ must be zero.
\end{proof}

\section{large volume limit and support property}\label{Support}

	Suppose we have a rational stability condition $\sigma=(\mathcal{A},Z)$ on $D(X)$. By Definition \ref{first WSC}, $Z$ factors as $K(\mathcal{A})=K(D(X))\xrightarrow{v} \Lambda\xrightarrow{g} \mathbb{C}$.  We assume $\sigma$ satisfies the support property with respect to the quadratic form $Q$ on $\Lambda\otimes\mathbb{R}$.
	
	There is an equivalent definition of support property.
	
	\begin{definition}[{\cite[Section 1.2]{kontsevich2008stability}}] Pick a norm $\|\underbar{\ }\|$ on $\Lambda\otimes\mathbb{R}$. The stability condition $\sigma$ satisfies the support property if there exists a constant $C>0$ such that for all $\sigma$-semistable objects $0\neq E\in D(X)$, we have $$\|v(E)\|\leq C|Z(E)|.$$ 
		
		Then the quadratic form $Q$ can be written as $Q(w)\coloneqq C^2|Z(w)|^2-\|w\|^2.$
		
	\end{definition}
	
	We have the following factorization of $Z_S^{s,t}$.
	\begin{lemma}

	The central charge $Z_S^{s,t}$ factors as $$K(\mathcal{A}_S)\xrightarrow{(v_1,v_2)^T} \Lambda\oplus\Lambda\xrightarrow{(sIm(g)-itRe(g), g)} \mathbb{C},$$ where $$v_1(E)=v(p_*(E\otimes q^*\mathcal{O}(n)))-v(p_*(E\otimes\ q^*\mathcal{O}(n-1))),$$  $$v_2(E)=v(p_*(E\otimes q^*\mathcal{O}(n)))-n\cdot v_1(E)$$ for $n\gg 0$,  and the superscript $T$ in $(v_1,v_2)^T$ stands for transpose. We also have $Im(g)\circ v_1=c$, $Re(g)\circ v_1=a$, $Im(g)\circ v_2=d$, $Re(g)\circ v_2=b$.
	
\end{lemma}
\begin{proof}
Assume we have a short exact sequence $$0\rightarrow K\rightarrow E\rightarrow Q\rightarrow 0$$ in $\mathcal{A}_S$. This gives a triangle $$ 	p_*(K\otimes q^*\mathcal{O}(n))\rightarrow p_*(E\otimes q^*\mathcal{O}(n))\rightarrow p_*(Q\otimes q^*\mathcal{O}(n))\xrightarrow{[1]}p_*(K\otimes q^*\mathcal{O}(n))[1]$$ in $D(X)$. Since $v:K(\mathcal{A})=K(D(X))\rightarrow\Lambda$ is a group homomorphism, we get $$v(p_*(E\otimes q^*\mathcal{O}(n)))=v(p_*(K\otimes q^*\mathcal{O}(n)))+v(p_*(Q\otimes q^*\mathcal{O}(n)))$$ for any $n\gg 0$. This proves that $v_1(E)=v_1(K)+v_1(Q)$. Therefore, $v_1$ is a group homomorphism. Similarly, we can prove $v_2$ is a group homomorphism too. 

We also need to prove that $v_1$ and $v_2$ are independent of $n$ when $n\gg 0$. For $v_1$, the independence follows from the following triangle $$p_*(E\otimes q^*\mathcal{O}(n-1))\rightarrow p_*(E\otimes q^*\mathcal{O}(n))\rightarrow p_*(E|_D)\xrightarrow{[1]}p_*(E\otimes q^*\mathcal{O}(n-1))[1],$$ where $D\in|\mathcal{O}(1)|$ is a $0$ dimensional subscheme of finite length and $E|_D$ is the derived pull-back of $E$ along $X\times D$. Hence we have $v_1(E)=v(p_*(E|_D))$, which is independent of $n$.

The independence of $v_2$ follows from the independence of $v_1$ by simple calculation. 

Apply $g$ on $v_1$, we get $g\circ v_1(E)=g\circ v(p_*(E\otimes q^*\mathcal{O}(n)))-g\circ v(p_*(E\otimes q^*\mathcal{O}(n-1)))=L_E(n)-L_E(n-1)=a(E)+ic(E)$. Hence we get $Re(g)\circ v_1=a$ and $Im(g)\circ v_1=c$. Similarly, we can prove that $Im(g)\circ v_2=d$, $Re(g)\circ v_2=b$. 

Therefore, we have $$(sIm(g)-itRe(g), g)\circ(v_1,v_2)^T(E)=sc(E)-ita(E)+b(E)+id(E)=Z_S^{s,t}(E).$$
\end{proof}
	\begin{definition}
		We call $w\in\Lambda$ a semi-stable vector if $w=v(E)$ for some semi-stable object $E\in\mathcal{A}$.
	\end{definition}
	
	\begin{lemma}\label{quadratic form}
		If $E\in\mathcal{A}_S$ is semi-stable with respect to $\mu_1$, then $v_1(E)$ is a semi-stable vector.
	\end{lemma}
	\begin{proof}
		We can take the short exact sequence $$0\rightarrow F\rightarrow E\rightarrow \bar{E}\rightarrow 0$$ where $F$ is the maximal torsion subobject of $E$, $\bar{E}$ is torsion free hence t-flat. It is easy to see that $v_1(E)=v_1(\bar{E})$. Therefore, we can assume $E$ is t-flat. 
		
		Now we consider the short exact sequence in Proposition \ref{semistable reduction} $$0\rightarrow K\rightarrow E\rightarrow Q\rightarrow 0.$$ Here $Q$ is torsion and $K$ is torsion free since $E$ is torsion free, hence $K$ is t-flat and semi-stable with respect to $\mu_2$. By Proposition \ref{fiberwise semistable}, $K$ is fiberwisely semi-stable. Therefore, it is easy to see $v_1(E)=v_1(K)$ is a semi-stable vector.

	\end{proof}
	\begin{lemma}\label{basic support}
		If $E\in\mathcal{A}_S$ is semi-stable with respect to the weak stability condition $\sigma_t=(\mathcal{A}_S,Z_t)$ for a fixed $t\in\mathbb{Q}_{>0}$, then $$b(E)c(E)-a(E)d(E)+\eta Q(v_1(E))\geq 0$$ for $0\leq \eta\leq \frac{t}{C^2}$.
	\end{lemma}
	\begin{proof}
		If $c(E)=0$, then $b(E)c(E)-a(E)d(E)\geq 0$ by Lemma \ref{easy lemma} and $Q(v_1(E))\geq 0$ by last lemma. Therefore, the statement is true in this case.
		
	We assume $c(E)>0$,  take the HN filtration $$0=E_0\subset E_1 \subset E_2\subset \cdots \subset E_{l-1} \subset E_l=E,$$ and use $Q_k$ to denote $E_{k}/E_{k-1}$ for $1\leq k\leq l$. Then we have the same inequalities (1) and (2) as in the proof of  Theorem \ref{the construction}, and we have the following inequalities
		
		\begin{equation}
			\frac{\Sigma_{k=1}^j (-a(Q_k)t+d(Q_k))}{\Sigma_{k=1}^jc(Q_k)t}\leq \frac{-a(E)t+d(E)}{c(E)t}\leq \frac{\Sigma_{k=j}^l(-a(Q_k)t+d(Q_k))}{\Sigma_{k=j}^lc(Q_k)t}.
		\end{equation}
		We have $c(Q_k)>0$ by same reason in the proof of Theorem \ref{the construction}. Similarly we get
		\begin{equation*}
			\begin{split}
				b(E) &\geq \Sigma_{k=1}^l \frac{a(Q_k)d(Q_k)}{c(Q_k)} \\ & =\frac{a(Q_l)}{c(Q_l)}d(F)-\Sigma_{j=1}^{l-1}\Sigma_{k=1}^jd(Q_k)(\frac{a(Q_{j+1})}{c(Q_{j+1})}-\frac{a(Q_j)}{c(Q_j)} )\\ & \geq \frac{a(Q_l)}{c(Q_l)}d(F)-\Sigma_{j=1}^{l-1}(\Sigma_{k=1}^j\frac{-a(E)t+d(E)}{c(E)}c(Q_k)+\Sigma_{k=1}^ja(Q_k)t)(\frac{a(Q_{j+1})}{c(Q_{j+1})}-\frac{a(Q_j)}{c(Q_j)} )\\ &  =\Sigma_{k=1}^l\frac{a(Q_k)}{c(Q_k)}(\frac{-a(E)t+d(E)}{c(E)}c(Q_k)+a(Q_k)t) \\& =a(E)\frac{-a(E)t+d(E)}{c(E)}+\Sigma_{k=1}^l\frac{a(Q_k)^2t}{c(Q_k)}.
			\end{split}
		\end{equation*}
		Therefore, we have 
		\begin{equation*}
			\begin{split}
				b(E)c(E)-a(E)d(E) &\geq \Sigma_{k=1}^lc(Q_k)\Sigma_{k=1}^l\frac{a(Q_k)^2t}{c(Q_k)}-a(E)^2t\\ & =t\Sigma_{1\leq i<j\leq l}(\frac{a(Q_i)}{\sqrt{c(Q_i)}}\sqrt{c(Q_j)}-\frac{a(Q_j)}{\sqrt{c(Q_j)}}\sqrt{c(Q_i)})^2.
			\end{split}	
		\end{equation*}
		
		On the other hand, let $w_i\coloneqq v_1(Q_i)$. Then $$Q(v_1(E))=Q(\Sigma_{i=1}^lw_i)=\Sigma_{i=1}^lQ(w_i)+2\Sigma_{1\leq i<j\leq l}Q(w_i,w_j).$$ By Lemma \ref{quadratic form}, we know that $Q(w_i)\geq 0$. Therefore, it suffices to prove that $$t(\frac{a(Q_i)}{\sqrt{c(Q_i)}}\sqrt{c(Q_j)}-\frac{a(Q_j)}{\sqrt{c(Q_j)}}\sqrt{c(Q_i)})^2+2\eta Q(w_i,w_j)\geq 0$$ for $0\leq \eta\leq \frac{t}{C^2}$ and arbitrary $1\leq i<j\leq l$.

		We have 
		\begin{equation*}
			\begin{split}
				2Q(w_i,w_j) &=Q(w_i+w_j)-Q(w_i)-Q(w_j) \\ & =C^2|Z(Q_i)+Z(Q_j)|^2-\|w_i+w_j\|^2-C^2|Z(Q_i)|^2+\|w_i\|^2-C^2|Z(Q_j)|^2+\|w_j\|^2\\ & =C^2((a(Q_i)+a(Q_j))^2+(c(Q_i)+c(Q_j))^2)-a(Q_i)^2-c(Q_i)^2-a(Q_j)^2-c(Q_j)^2) \\&-\|w_i+w_j\|^2+\|w_i\|^2+\|w_j\|^2 \\ & \geq 2C^2(a(Q_i)a(Q_j)+c(Q_i)c(Q_j))-2\|w_i\|\|w_j\| \\ &\geq 2C^2(a(Q_i)a(Q_j)+c(Q_i)c(Q_j)-\sqrt{a(Q_i)^2+c(Q_i)^2}\sqrt{a(Q_j)^2+c(Q_j)^2}).
			\end{split}	
		\end{equation*}
		The first inequality is from Cauchy-Schwarz inequality. The second inequality comes from the definition of support property. Hence, it suffices to prove that 
		\begin{equation*}
			\begin{split}
				t(\frac{a(Q_i)}{\sqrt{c(Q_i)}}\sqrt{c(Q_j)}&- \frac{a(Q_j)}{\sqrt{c(Q_j)}}\sqrt{c(Q_i)})^2\\ &\geq 
				2\eta C^2(\sqrt{a(Q_i)^2+c(Q_i)^2}\sqrt{a(Q_j)^2+c(Q_j)^2}-a(Q_i)a(Q_j)-c(Q_i)c(Q_j)).
			\end{split}	
		\end{equation*}
		By Cauchy's inequality, the right hand side is nonnegative. Hence, it is enough to prove the inequality for $\eta=\frac{t}{C^2}$. In this case, the inequality is equivalent to 
		\begin{equation*}
			\begin{split}
				\frac{a(Q_i)^2}{c(Q_i)}c(Q_j)&+\frac{a(Q_j)^2}{c(Q_j)}c(Q_i)+2c(Q_i)c(Q_j)\\ &\geq 
				2\sqrt{a(Q_i)^2+c(Q_i)^2}\sqrt{a(Q_j)^2+c(Q_j)^2}
			\end{split}	
		\end{equation*}
		which becomes trivial if we divide both sides by $c(Q_i)c(Q_j)$. Therefore, the lemma is proved.
	\end{proof}

	Now, we consider the stability conditions $\sigma_{s,t}=(\mathcal{A}_S^t,Z_S^{s,t})$.
	\begin{lemma}\label{middle support}
		If $t$ is a fixed positive rational number and $E\in \mathcal{A}_S^t$ is semi-stable with respect to $\sigma_{s,t}=(\mathcal{A}_S^t,Z_S^{s,t})$ for all $s$ sufficiently large. Then $b(E)c(E)-a(E)d(E)+\eta Q(v_1(E))\geq 0$ for $0\leq \eta\leq \frac{t}{C^2}.$
	\end{lemma}
	
	\begin{proof}
		We have that $E$ sits in the following short exact sequence

		$$0\rightarrow F[1]\rightarrow E\rightarrow T\rightarrow 0$$where $F\in\mathcal{F}, T\in\mathcal{T}$. We claim that one of the following cases is true:
		
		(i) $F=0$, and $T$ is semi-stable with respect to $\sigma_t$.
		
		(ii) $c(T)=a(T)=d(T)=0$, and $F$ is semi-stable with respect to $\sigma_t$.
		
		The proof of the claim is essentially the same as in \cite[Lemma 8.9]{bayer2016space}. We sketch the proof for reader's convenience.
		
		If $F=0$, it is easy to see $T$ is semi-stable with respect to $\sigma_t$. Therefore, the inequality holds by Lemma \ref{basic support}.
		
		Now we can assume $F\neq 0$. Since $E$ is semi-stable with respect to $\sigma_{s,t}=(\mathcal{A}_S^t,Z_s^{s,t})$ for $s$ sufficiently large. We have 
		$$\frac{c(F)s+b(F)}{a(F)t-d(F)}\leq \frac{(c(F)-c(T))s+b(F)-b(T)}{(-a(T)+a(F))t+d(T)-d(F)}$$
		for $s$ sufficiently large. 
		
		But we have $$(-a(T)+a(F))t+d(T)-d(F)\geq a(F)t-d(F)\geq 0,$$ and $c(F)>0,\ c(T)\geq 0$. As a consequence, the inequality can hold if only if $c(T)=a(T)=d(T)=0$, which implies that $T$ is a torsion object, hence $v_1(T)=0$. In this case, it is easy to see $F$ is semi-stable with respect to $\sigma_t$.
		
		Then \begin{equation*}
			\begin{split}
				b(E)c(E)-a(E)d(E)&=(b(T)-b(F))(-c(F))-(-a(F))(-d(F))\\&=-b(T)c(F)+b(F)c(F)-a(F)d(F)\\&\geq b(F)c(F)-a(F)d(F)
			\end{split}
		\end{equation*} and $Q(v_1(E))=Q(v_1(F))$ because $v_1(T)=0$. Hence, the inequality follows from Lemma \ref{basic support}.
	\end{proof}
	
	\begin{lemma}\label{support property}
		If $E\in \mathcal{A}_S^t$ is semi-stable with respect to $\sigma_{s,t}=(\mathcal{A}_S^t,Z_S^{s,t})$ for fixed $s,t\in\mathbb{Q}_{>0}$. Then $$b(E)c(E)-a(E)d(E)+\eta Q(v_1(E))\geq 0$$ for any $0\leq \eta\leq \frac{min\{s,t\}}{C^2}$.
	\end{lemma}
	
	\begin{proof}
		The idea of this proof is essentially from \cite[Lemma 8.8]{bayer2016space}.
		We first prove that the kernel of $Z_S^{s,t}$ is negative semi-definite with respect to $$bc-ad+\eta Q$$ for $0\leq \eta\leq \frac{min\{s,t\}}{C^2}$.
		
		If $Z_S^{s,t}(E)=0$, we have \begin{equation*}
			\begin{split}
				b(E)c(E)-a(E)d(E)+\eta Q(v_1(E))&=-sc(E)^2-ta(E)^2+\eta Q(v_1(E))\\ &\leq -sc(E)^2-ta(E)^2+\eta C^2(a(E)^2+c(E)^2)\leq 0
			\end{split}
		\end{equation*}
		for $0\leq \eta\leq \frac{min\{s,t\}}{C^2}$.

		Since the image of $-at+d$ is discrete, we can prove the lemma by induction on $-a(E)t+d(E)$. If $-a(E)t+d(E)=0$ or $-a(E)t+d(E)$ is minimal in the image of imaginary part, then it is easy to see $E$ is semi-stable for $m$ sufficiently large. Therefore, the inequality holds by Lemma \ref{middle support}. 
		Now we assume it is true for objects whose imaginary part is less than $N_0>0$. \
		
		Let $t$ be a positive rational number. Suppose $E$ is semi-stable with respect to $\sigma_{s_0,t}$ for a positive rational number $s_0$ and $-a(E)t+d(E)=N_0$. By \cite[Lemma A.6]{bayer2016space}, we can assume $E$ is stable. If $E$ remains semi-stable with respect to $\sigma_{s,,t}$, for all $s>s_0$, then this follows from Lemma \ref{middle support}. 
		
		Otherwise, suppose $E$ is unstable with respect to $\sigma_{s_1,t}$ for a positive rational number $s_1$ bigger than $s_0$. Let $$W\coloneqq\{Z_S^{s_1,t}(F)|0\neq F\subset E\ and\  \mu_{s_1,t}(F)>\mu_{s_1,t}(E)\}$$ where $\mu_{s_1,t}$ is the associated slope function of $\sigma_{s_1,t}$. Then by \cite[Lemma 4.9]{macri2017lectures} and discreteness of $Z_S^{s_1,t}$, we know that $W$ is a finite subset in $\mathbb{C}$. 
		
		For any element $w\in W$, we set $$\mathcal{M}_w\coloneqq\{F\mid F\subset E\ and\ Z^{s_1,t}_S(F)=w\}.$$ 
		
		Then since $Z_S^{s_0,t}$ is discrete, we can find $F_w\in\mathcal{M}_w$ such that $$\mu_{s_0,t}(F_w)=\max_{F\in\mathcal{M}_w}\mu_{s_0,t}(F).$$
		
		Since $a,b,c,d$ are linear and rational, we can find a positive rational number $s_w$, such that $s_0<s_w<s_1$ and $$Z_S^{s_w,t}(F_w)/Z_S^{s_w,t}(E)\in\mathbb{R}_{>0}.$$ 
		
		Moreover, by the definition of $F_w$, we know that $\mu_{s_w,t}(F)\leq\mu_{s_w,t}(E)$ for all $F\in\mathcal{M}_w$. Therefore, if we take $s'=\min_{w\in W}s_w$, which is also a positive rational number and $s_0<s'<s_1$. One can easily check that $E$ is strictly semi-stable with respect to $\sigma_{s',t}$. Then by induction, all its Jordan-H\"older factor (with respect to $\sigma_{s',t}$) satisfy the inequality. Therefore, the inequality holds for $E$ by \cite[Lemma A.6]{bayer2016space}.
		
	\end{proof}
	\begin{remark}
		Sometimes we only use the case $\eta=0$, like in the following theorem. The difference is that when $\eta=0$, we only get the support property on  a rank 4 quotient lattice $\Lambda/ker(g)\oplus \Lambda/ker(g)$. While for $0< \eta<\frac{min\{s,t\}}{C^2}$, we get the stronger support property on the lattice $\Lambda\oplus \Lambda/ker(g)$.  Indeed, for $E\in ker(Z_S^{s,t})$, we have $$b(E)c(E)-a(E)d(E)=-sc(E)^2-ta(E)^2\leq 0.$$ The equality holds if and only if $$a(E)=b(E)=c(E)=d(E)=0.$$ While for $0< \eta<\frac{min\{s,t\}}{C^2}$, we have \begin{equation*}
		\begin{split}
			b(E)c(E)-a(E)d(E)+\eta Q(v_1(E))&=-sc(E)^2-ta(E)^2+\eta Q(v_1(E))\\ &\leq -sc(E)^2-ta(E)^2+\eta C^2(a(E)^2+c(E)^2)\leq 0.
		\end{split}
	\end{equation*} The equality holds if and only if $a(E)=b(E)=c(E)=d(E)=0$ and $v_1(E)=0$ since $Q|_{ker(g)}$ is negative definite and $v_1(E)\in ker(g)$.
	\end{remark}
	\begin{theorem}
		We have a map $\eta:RStab(X)\times\mathbb{R}_{> 0}\times\mathbb{R}_{> 0}\rightarrow Stab(X\times S)$, where $S$ is an integral smooth projective curve. Moreover, the stability conditions in the image satisfy the support property.
	\end{theorem}
	
	\begin{proof}
		The image satisfy the support property because of Lemma \ref{support property}.
		
		By Theorem \ref{map of rational stability conditions}, we have the map $\eta':RStab(X)\times \mathbb{Q}_{>0}\times \mathbb{Q}_{>0}\rightarrow Stab(X\times S)$. We only need to prove that $\eta'$ is continuous at the factor $\mathbb{Q}_{>0}\times \mathbb{Q}_{>0}$ for any given rational stability condition $\sigma\in RStab(X)$.
		
		We look at rational stability conditions $\sigma_{s_0,t_0}$, where $s_0,t_0\in\mathbb{Q}_{>0}$. We assume that $s_0\geq t_0$ (the case $s_0<t_0$ is similar). Then if $|s-s_0|<sin(\frac{1}{10}\pi)t_0$ and $|t-t_0|<sin(\frac{1}{10}\pi)t_0$, we have 
		
		$$|Z_S^{s,t}(E)-Z_S^{s_0,t_0}(E)|\leq sin(\frac{1}{10}\pi)|Z_S^{s_0,t_0}(E)|$$for $E$ semi-stable with respect to $\sigma_{s_0,t_0}$. More specifically, by Lemma \ref{support property}, we have 
		
		\begin{equation*}
			\begin{split}
				|Z_S^{s_0,t_0}(E)|\ &=(c(E)^2s_0^2+b(E)^2+2s_0b(E)c(E)+a(E)^2t_0^2+d(E)^2-2t_0a(E)d(E))^{\frac{1}{2}} \\ &
				\geq (c(E)^2(s_0^2-(s_0-t_0)^2)+a(E)^2t_0^2+d(E)^2+2t_0(b(E)c(E)-a(E)d(E)))^{\frac{1}{2}} \\ &
				\geq (c(E)^2t_0^2+a(E)^2t_0^2)^{\frac{1}{2}}
			\end{split}
		\end{equation*}
		and \begin{equation*}
			\begin{split}
				|Z_S^{s,t}(E)-Z_S^{s_0,t_0}(E)|&=((t-t_0)^2a(E)^2+(s-s_0)^2c(E)^2)^{\frac{1}{2}}\\&\leq sin(\frac{1}{10}\pi)(c(E)^2t_0^2+a(E)^2t_0^2)^{\frac{1}{2}}.
			\end{split}
		\end{equation*}

		Therefore, by \cite[Theorem 7.1]{bridgeland2007stability}, we get the map.
		
	\end{proof}

\bibliographystyle{alpha}
\bibliography{bibfile}

\newcommand{\etalchar}[1]{$^{#1}$}
\begin{thebibliography}{BLM{\etalchar{+}}19}

\bibitem[AB13]{ABsurfaces}
Daniele Arcara and Aaron Bertram.
\newblock Bridgeland-stable moduli spaces for {$K$}-trivial surfaces.
\newblock {\em J. Eur. Math. Soc. (JEMS)}, 15(1):1--38, 2013.
\newblock With an appendix by Max Lieblich.

\bibitem[AP06]{APsheaves}
Dan Abramovich and Alexander Polishchuk.
\newblock Sheaves of {$t$}-structures and valuative criteria for stable
  complexes.
\newblock {\em J. Reine Angew. Math.}, 590:89--130, 2006.

\bibitem[BLM{\etalchar{+}}19]{bayer2019stability}
Arend Bayer, Mart{\'\i} Lahoz, Emanuele Macr{\`\i}, Howard Nuer, Alexander
  Perry, and Paolo Stellari.
\newblock Stability conditions in families.
\newblock {\em arXiv preprint arXiv:1902.08184}, 2019.

\bibitem[BLMS17]{bayer2017stability}
Arend Bayer, Mart{\'\i} Lahoz, Emanuele Macr{\`\i}, and Paolo Stellari.
\newblock Stability conditions on {K}uznetsov components.
\newblock {\em arXiv preprint arXiv:1703.10839}, 2017.

\bibitem[BM14a]{bayer2014mmp}
Arend Bayer and Emanuele Macr\`\i.
\newblock M{MP} for moduli of sheaves on {K}3s via wall-crossing: nef and
  movable cones, {L}agrangian fibrations.
\newblock {\em Invent. Math.}, 198(3):505--590, 2014.

\bibitem[BM14b]{bayer2014projectivity}
Arend Bayer and Emanuele Macr\`\i.
\newblock Projectivity and birational geometry of {B}ridgeland moduli spaces.
\newblock {\em J. Amer. Math. Soc.}, 27(3):707--752, 2014.

\bibitem[BMS16]{bayer2016space}
Arend Bayer, Emanuele Macr\`\i, and Paolo Stellari.
\newblock The space of stability conditions on abelian threefolds, and on some
  {C}alabi-{Y}au threefolds.
\newblock {\em Invent. Math.}, 206(3):869--933, 2016.

\bibitem[BMT14]{baye2011bridgeland}
Arend Bayer, Emanuele Macr\`\i, and Yukinobu Toda.
\newblock Bridgeland stability conditions on threefolds {I}:
  {B}ogomolov-{G}ieseker type inequalities.
\newblock {\em J. Algebraic Geom.}, 23(1):117--163, 2014.

\bibitem[Bri07]{bridgeland2007stability}
Tom Bridgeland.
\newblock Stability conditions on triangulated categories.
\newblock {\em Ann. of Math. (2)}, 166(2):317--345, 2007.

\bibitem[Bri08]{bridgeland2008stability}
Tom Bridgeland.
\newblock Stability conditions on {$K3$} surfaces.
\newblock {\em Duke Math. J.}, 141(2):241--291, 2008.

\bibitem[Dou02]{douglas2002dirichlet}
Michael~R. Douglas.
\newblock Dirichlet branes, homological mirror symmetry, and stability.
\newblock In {\em Proceedings of the {I}nternational {C}ongress of
  {M}athematicians, {V}ol. {III} ({B}eijing, 2002)}, pages 395--408. Higher Ed.
  Press, Beijing, 2002.

\bibitem[HRS96]{happel1996tilting}
Dieter Happel, Idun Reiten, and Sverre~O. Smal\o.
\newblock Tilting in abelian categories and quasitilted algebras.
\newblock {\em Mem. Amer. Math. Soc.}, 120(575):viii+ 88, 1996.

\bibitem[Huy06]{huybrechts2006fourier}
D.~Huybrechts.
\newblock {\em Fourier-{M}ukai transforms in algebraic geometry}.
\newblock Oxford Mathematical Monographs. The Clarendon Press, Oxford
  University Press, Oxford, 2006.

\bibitem[Kon15]{Kontsevichstability}
Maxim Kontsevich.
\newblock Iterated stability.
\newblock talk at {AMS} {S}ummer {I}nstitute in {A}lgebraic {G}eometry, 2015.
\newblock Accessed 2019-01-31.

\bibitem[Kos18]{Koseki}
Naoki Koseki.
\newblock Stability conditions on product threefolds of projective spaces and
  {A}belian varieties.
\newblock {\em Bull. Lond. Math. Soc.}, 50(2):229--244, 2018.

\bibitem[KS08]{kontsevich2008stability}
Maxim Kontsevich and Yan Soibelman.
\newblock Stability structures, motivic {D}onaldson-{T}homas invariants and
  cluster transformations.
\newblock {\em arXiv preprint arXiv:0811.2435}, 2008.

\bibitem[Laz04]{lazarsfeld2004positivity}
Robert Lazarsfeld.
\newblock {\em Positivity in algebraic geometry. {I}}, volume~48 of {\em Ergeb.
  Math. Grenzgeb. (3).}
\newblock Springer-Verlag, Berlin, 2004.

\bibitem[Li19]{li2018stability}
Chunyi Li.
\newblock On stability conditions for the quintic threefold.
\newblock {\em Invent. Math.}, 218(1):301--340, 2019.

\bibitem[MP15]{MPabelian}
Antony Maciocia and Dulip Piyaratne.
\newblock Fourier-{M}ukai transforms and {B}ridgeland stability conditions on
  abelian threefolds.
\newblock {\em Algebr. Geom.}, 2(3):270--297, 2015.

\bibitem[MS17]{macri2017lectures}
Emanuele Macr{\`\i} and Benjamin Schmidt.
\newblock Lectures on {B}ridgeland stability.
\newblock In {\em Moduli of curves}, pages 139--211. Springer, 2017.

\bibitem[Pol07]{polishchuk2007constant}
A.~Polishchuk.
\newblock Constant families of {$t$}-structures on derived categories of
  coherent sheaves.
\newblock {\em Mosc. Math. J.}, 7(1):109--134, 167, 2007.

\bibitem[PT19]{piyaratne2019moduli}
Dulip Piyaratne and Yukinobu Toda.
\newblock Moduli of {B}ridgeland semistable objects on 3-folds and
  {D}onaldson-{T}homas invariants.
\newblock {\em J. Reine Angew. Math.}, 747:175--219, 2019.

\bibitem[Sun19a]{Sun19inequality}
Hao Sun.
\newblock Bogomolov's inequality for product type varieties in positive
  characteristic.
\newblock {\em arXiv preprint arXiv:1907.08378}, 2019.

\bibitem[Sun19b]{Sun19stability}
Hao Sun.
\newblock Bridgeland stability conditions on some threefolds of general type.
\newblock {\em arXiv preprint arXiv:1907.08379}, 2019.

\end{thebibliography}

\end{document}